\documentclass[generic]{imsart}
\usepackage{enumerate}
\usepackage{amsthm}
\usepackage{amsfonts}
\usepackage{amsmath}
\usepackage{amssymb}
\usepackage{bm}

\RequirePackage[T1]{fontenc}
\RequirePackage{amsthm,amsmath}
\usepackage{amsfonts}
\RequirePackage[numbers]{natbib}
\RequirePackage[colorlinks,citecolor=blue,urlcolor=blue]{hyperref}
\usepackage[utf8]{inputenc}

\pubyear{0000}
\volume{0}
\issue{0}
\firstpage{0}
\lastpage{0}

\startlocaldefs
\numberwithin{equation}{section}
\theoremstyle{plain}
\newtheorem{thm}{Theorem}[section]

\newtheorem{remk}{Remark}[section]
\newtheorem{prop}{Proposition}[section]

\newtheorem{lem}{Lemma}[section]
\newtheorem{corol}{Corollary}[section]

\endlocaldefs

\DeclareMathOperator*{\rank}{rank}

\def\Det#1{\left|#1\right|}

\begin{document}

\begin{frontmatter}

\title{
On the maximum likelihood degree of linear mixed models with two variance components}

\runtitle{Linear mixed  models with two variance components}

\begin{aug}
\author{\fnms{Mariusz} \snm{{Grządziel}\ead[label=e1]{mariusz.grzadziel@up.wroc.pl}}}

\address{
Department of Mathematics,\\
Wrocław University of Environmental and Life Sciences\\
Grunwaldzka 53, 50 357 Wrocław, Poland
\printead{e1}}

\runauthor{M.~Grządziel}

\affiliation{Wroc\l aw University of Environmental and Life Sciences}

\end{aug}

\begin{abstract}
We extend 
the results concerning the upper bounds for the maximum likelihood degree and the REML degree of the one-way random effects model presented in Gross \textit{et al.} [Electron. J. Stat. 6 (2012), pp.~993--1016] 
to the case of the normal linear mixed model 
with two variance components.
Then we 
prove that both parts of Conjecture~1 in the paper of Gross \textit{et al.}, which 
concerns a certain extension of the one-way random effects model, are true under fairly mild conditions. 
\end{abstract}


\tableofcontents
\end{frontmatter}

\section{Introduction}
The notion of the maximum likelihood (ML) degree of the statistical model
that has rational ML equations 
was introduced in \cite{CHKS06} and
\cite{HKS05}.
It can be defined as `the number of complex solutions to the likelihood equations when data are
generic'; a data set is generic `if it is not a part of the null set for which the number of complex solutions is different' (compare \cite[p.~995]{GDP12} and \cite[Chapter~2]{DSS08}). It may be interpreted as a measure of the computational complexity of the problem of solving the ML equations algebraically.

In the paper \cite{GDP12}, the authors considered 
the ML degree of variance component models. They also introduced the notion of the
restricted maximum likelihood 
(REML) degree of these models. The REML estimator of the vector of variance components 
can be defined as its ML estimator in the
appropriately modified model; the REML degree can be defined as the ML degree of this modified model.
The authors of \cite{GDP12} computed the ML degree and the REML degree for the one-way  random effects model
and formulated a conjecture concerning upper bounds for the ML degree and the REML degree of its extension, 
in which the mean structure is more general.

In this paper, we study the  linear mixed model with two variance components. 
We give upper bounds for the ML degree and for the REML degree of this model.
As a corollary, we obtain  conditions under which the part of  Conjecture~1 in \cite{GDP12}, which concerns the upper bound for the ML degree of the
mentioned earlier extension of the
one-way random effects model, holds, and the similar conditions  for the part of the conjecture
concerning the
REML degree of this model.
The obtained bounds reduce to those given in \cite[p.~996]{DSS08} in the case of the one-way random effects model.

The remaining sections of the paper are organized as follows. Basic facts
concerning the linear mixed model with two variance components are gathered in
Section \ref{section:ML_REML}.
The theorems giving upper bounds for the ML degree and the REML degree of this model are presented in Section \ref{sec:ML_Degree}
and in Section \ref{sec:REML_Degree}. The conditions under which the statements of the conjecture from \cite{GDP12} are true are discussed in Section \ref{section:GeneralMean}. The final conclusion is presented in Section \ref{sec:Conclusion}.

The following notation will be used: for a
matrix $A$,  $A^{\prime}$ will denote the transpose of $A$,
$A^{+}$ the Moore-Penrose inverse of $A$,
$\rank(A)$ the rank of $A$
and  $\mathcal{M}(A)$ the space spanned by the columns of $ A $.
The determinant of a square matrix $A$ will be denoted by
$\Det A$.
The notation $ [A_{1},A_{2}] $ will be used for a partitioned matrix consisting of two blocks:
$ A_{1} $ and $ A_{2} $.
The symbol
$I_{n}$ will stand for the 
identity matrix of order $n$,
$ \mathbf{0}_{n} $
for the $n$-dimensional column vector of zeros and
$ \mathbf{1}_{n} $
for the $n$-dimensional column vector of ones.

\section{Model with two variance components}
\label{section:ML_REML}

Let us consider the normal 
linear mixed model with two variance components
in which the dependent variable $ Y $ is an $n \times 1$ normally distributed vector with
\begin{equation} \label{model}
E(Y)=X\beta, \quad Cov(Y)=\Sigma(s)=\sigma_{1}^{2}V+\sigma_{2}^{2}I_{n},      
\end{equation}
where
$X$ is an $n \times p$ matrix of full rank, $p<n$,
$\beta$ is a $p \times 1$ parameter vector,
$ V $ is an $n \times n$ non-negative definite symmetric non-zero matrix of rank $k<n$
and $ s=(\sigma_{1}^{2},\sigma_{2}^{2})^{\prime}$ 
is an unknown vector of variance components belonging to
$ \mathcal{S}= \{s: \sigma_{1}^{2} \ge 0,\sigma_{2}^{2}>0\}$.
We will denote this model by
$\mathcal{N}(Y, X \beta, \Sigma(s))$;
for more information concerning this approach to representing the linear mixed models, see \cite[p.~7]{Jia07}.

The twice the log-likelihood function is given, up to an additive constant, by 
\begin{equation} \label{eq:Loglikelihood}
l_{0}(\beta,s,Y):=-\log\Det{\Sigma(s)}-(Y-X\beta)^{\prime}\Sigma^{-1}(s)(Y-X\beta). 
\end{equation}
The ML estimator of $ (\beta,s) $ is defined as the maximizer of $ l_{0}(\beta,s,Y) $ over $ (\beta,s) \in \mathbb{R}^{n} \times \mathcal{S}$. 

Let  $M:=I_{n}-XX^{+}$. It can be shown that
\[
l_{0}(\beta,s,Y) \le l_{0}(\tilde \beta,s,Y)=-\log\Det{\Sigma(s)} - Y^{\prime} R(s)Y,
\]
where $R(s):= (M\Sigma(s)M)^{+}$
and $\tilde \beta(s):=(X^{\prime}\Sigma^{-1}(s)X)^{-1}X^{\prime}\Sigma^{-1}(s)Y$. It can be checked that $l_{0}(\beta,s,Y) < l_{0}(\tilde \beta,s,Y)  $ for $\beta \neq \tilde \beta $.
It can be thus seen 
that the problem of computing the ML estimator of $ (\beta,s) $ reduces to finding the maximizer of $l(s,Y):=-\log\Det{\Sigma(s)} - Y^{\prime} R(s)Y$
over $s \in \mathcal{S} $,
which we will refer to as the ML estimator of $ s $,
compare \cite[p.~230]{RK88} and \cite[284]{GSTU02}.  
It can be also observed that
for a given value $ y $ of the vector $ Y$ the ML estimate of $ s $ exists if and only if the ML estimate of $ (\beta,s)  $ exists.

It can be verified that
the complex solutions to
the ML equations obtained by equating
$ \partial l_{0}(\beta,s,Y)/\partial \beta $,
$ \partial l_{0}(\beta,s,Y)/\partial \sigma_{1}^{2} $ and $ \partial l_{0}(\beta,s,Y)/\partial \sigma_{2}^{2} $ to $ 0 $
(or to $ \mathbf{0}_p$)
are of the form
$ ( \tilde{\beta}(\hat{s}),\hat{s} ) $, where
$ \hat{s} $ is a complex solution to 
the system
\begin{equation} \label{eq:ML} 
\frac{\partial l(s,Y)}{\partial \sigma_{1}^{2}}=0, \quad \frac{\partial l(s,Y)}{\partial \sigma_{2}^{2}}=0.
\end{equation}
Here the partial derivatives are computed with respect to the real variables $\sigma_{1}^{2}$,
$\sigma_{2}^{2}  $ and $ \beta $.
It can be also checked that if $ \hat{s} $ is a solution to 
(\ref{eq:ML}), then $ ( \tilde{\beta}(\hat{s}),\hat{s} ) $ is a solution to the ML equations. 
The ML degree of the model (\ref{model}) is thus equal to the number of complex solutions  to the system (\ref{eq:ML}) when the data are  generic.

The necessary and sufficient condition for the existence of the ML estimate of $ s $ is given
in the following

\begin{thm}
For a given value $ y $ of the vector $ Y $, the ML estimate of
$ s $ exists if and only if
\begin{equation} \label{eq:condML}
y \notin \mathcal{M}([X,V]).
\end{equation}	
\end{thm}

\begin{proof}
The theorem is an immediate consequence of Theorem 3.1 in \cite{DM99}.
\end{proof}

\begin{corol} \label{cor:NullSet}
The ML estimate of $ s $ exists with probability $ 1 $ if and only if
\begin{equation} \label{cond:ML_prob}
\mathcal{M}([X,V]) \subsetneq  \mathbb{R}^{n}.
\end{equation}
\end{corol}

Let $B$ be an $(n-p) \times n$ matrix satisfying the conditions
\begin{equation} \label{eq:B}
BB^{\prime}=I_{n-p}, \quad B^{\prime}B=M.
\end{equation}
The ML estimator of $s$ in the 
model  $\mathcal{N}(z,\mathbf{0}_{n-p},\sigma_{1}^{2}BVB^{\prime}+ \sigma_{2}^{2}I_{n-p})$,
where $z:=BY$,
is known in the literature as the restricted maximum likelihood (REML) estimator of $s$ (compare e.g. \cite[p.~291]{GSTU02} or \citep[p.~880]{OSB76}).

Let
\begin{equation} \label{eq:BVBprime}
BVB^{\prime}=\sum_{i=1}^{d-1}m_{i}E_{i}
\end{equation}
be the spectral decomposition of $BVB^{\prime} $, where $ m_{1}> \ldots > m_{d-1} >m_{d}=0$ 
denotes the decreasing sequence
of distinct
eigenvalues of $BVB^{\prime}$
and $ E_{i} $'s are orthogonal projectors satisfying the condition $E_{i}E_{j}=\mathbf{0}_{n-p}$ for $ i \neq j $. Let $ E_{d} $ be such that $\sum_{i=1}^{d}E_{i}=I_{n-p}$ and let
\begin{equation} \label{eq:Ti}
 T_{i}:=z^{\prime}E_{i}z/\nu_{i}, \quad z:=BY, 
\end{equation}
where $\nu_{i}$ denotes the multiplicity  of the eigenvalue $ m_{i}$, $ i=1,\ldots,d$. Let us note that
the quantities  $ m_{i}$, $E_{i}$, and $ \nu_{i}$, $ i=1,\ldots,d$, do not depend on the choice of $B$ in (\ref{eq:B}), see \cite[p.~880]{OSB76}. Throughout the paper we will assume that  
\begin{equation} \label{eq:nu}
\nu_{d}>0.
\end{equation}
The random variables $T_{i}$'s
are mutually independent and 
$\nu_{i}(m_{i} \sigma^{2}_{1} +\sigma_{2}^{2}     )^{-1} T_{i}$
has a central chi-squared distribution with $\nu_{i}$ degrees of freedom, $ i=1,\ldots,d$, compare \cite[p.~880]{OSB76}.

It can be shown that
\begin{thm}
	For a given value $ y $ of the vector $ Y $
in the model (\ref{model}),	
the REML estimate of
$ s $ exists if and only if
	\begin{equation} \label{eq:condREML}
      My \notin \mathcal{M}(MV).
	\end{equation}	
\end{thm}

\begin{proof}
	The theorem follows directly from Theorem 3.4 in \cite{DM99}.
 \end{proof}

\begin{prop}
The REML estimate of $s$ in the model (\ref{model}) exists with probability one if and only if
\begin{equation} \label{cond:REML_P1}
\mathcal{M}(MV) \subsetneq \mathcal{M}(M).
\end{equation}
\end{prop}

\begin{proof}
It can be seen that
$ \mathcal{M}(MV) \subsetneq \mathcal{M}(M) $
 if and only if 
$ \mathcal{M}(BV) \subsetneq \mathbb{R}^{{n-p} }$.
Since  $BY \sim N(\mathbf{0}_{n-p},\sigma_{1}^{2}BVB^{\prime}+ \sigma_{2}^{2}I_{n-p})$,
we have: $P(MY \notin \mathcal{M}(MV))=
P(BY \notin \mathcal{M}(BV))=1$
if and only if (\ref{cond:REML_P1})
holds.

\end{proof}

\begin{prop}
In the model (\ref{model}):

\begin{enumerate}[(a)]
\item
The condition (\ref{cond:ML_prob}) implies
the condition (\ref{cond:REML_P1}).
\item
The condition (\ref{cond:REML_P1})
is equivalent to the condition (\ref{eq:nu}).

\end{enumerate}
\end{prop}

\begin{proof}
To prove the part (a) let us assume that 
$\mathcal{M}(MV) = \mathcal{M}(M)$.
This condition implies that
$\mathcal{M}(M) \subset \mathcal{M}(V)$,
which is equivalent 
to $\mathcal{M}([X,V])=\mathbb{R}^{n}$, and so we have obtained a contradiction.

The part (b) is a consequence
of  the following
facts: 
the condition (\ref{cond:REML_P1}) is equivalent to 
$
\mathcal{M}(BV) \subsetneq \mathbb{R}^{n-p}
$
and $\mathcal{M}(BV)=\mathcal{M}(BVB^{\prime})$.
\end{proof}

\section{The ML degree of the model}
\label{sec:ML_Degree}

Let
$\alpha_{1}>\alpha_{2}>\ldots >\alpha_{d_{0}}=0$
denote
the decreasing sequence of distinct eigenvalues of $ V $
and let $s_{i}$, $i=1,\ldots,d_{0}$, stand for their multiplicities.
The ML equation system (\ref{eq:ML})
can be written as
\begin{align} \label{eq:MLEquations}
\begin{split}
  \sum_{i=1}^{d-1}\frac{\nu_{i}m_{i}}{(m_{i}\sigma_{1}^{2}+\sigma_{2}^{2})^{2}}T_{i} &=\sum_{j=1}^{d_{0}-1}\frac{s_{j}\alpha_{j}}{\alpha_{j}\sigma_{1}^{2}+\sigma_{2}^{2}},\\
\sum_{i=1}^{d}\frac{\nu_{i}}{(m_{i}\sigma_{1}^{2}+\sigma_{2}^{2})^{2}}T_{i}  &=\sum_{j=1}^{d_{0}}\frac{s_{j}}{\alpha_{j}\sigma_{1}^{2}+\sigma_{2}^{2}},
\end{split}
\end{align}
compare \cite[p.~286]{GSTU02}.

Let us introduce the variables
\[
\sigma^{2} := \sigma_{1}^{2}+\sigma_{2}^{2}, \quad   \rho: = \frac{\sigma_{1}^{2}}{\sigma_{1}^{2}+\sigma_{2}^{2}}. 
\]
It can be seen that
\begin{equation} \label{eq:inv}
\sigma_{1}^{2}=\sigma^{2}\rho,  \quad \sigma_{2}^{2}=\sigma^{2}(1-\rho).
\end{equation}
Let us define the rational algebraic expressions
\begin{align}
\phi_{\mu}(\rho) := &(\mu-1)\rho+1, \quad H_{1}(\rho):=\sum_{i=1}^{d-1} \frac{\nu_{i}m_{{i}}}{\phi^{2}_{m_{i}}(\rho)}T_{i}, \\
 H_{2}(\rho):= &\sum_{j=1}^{d_{0}-1}\frac{\alpha_{j}s_{j}}{\phi_{\alpha_{j}}(\rho)}, \quad \text{and} \quad h(\rho) := \frac{H_{1}(\rho)}{H_{2}(\rho)}.  \label{eq:h}
\end{align}

If $ T_{i}>0 $ for some $ i<d $, then the system (\ref{eq:MLEquations}),
assuming that $\sigma_{1}^{2} +  \sigma_{2}^{2} \neq 0$,
is equivalent to:
\begin{align}
\sigma^{2}&=h(\rho), \label{eq:sigma}  \\
\sum_{i=1}^{d}\frac{\nu_{i}}{\phi^{2}_{m_{i}}(\rho)}T_{i}&=h(\rho)\sum_{j=1}^{d_{0}}\frac{s_{j}}{\phi_{\alpha_{j}}
(\rho)}, \label{eq:rho}
\end{align}
compare \cite[p.~287]{GSTU02}.
Let us observe that if we find a solution to (\ref{eq:rho}) with respect to $ \rho $, we will be able to compute a solution to the system
(\ref{eq:MLEquations}) using (\ref{eq:sigma}) and (\ref{eq:inv}).
Let us also note that each solution to (\ref{eq:MLEquations}) that does not have the form $ (\theta, -\theta) $, where $ \theta $ is a complex number, can be obtained in this way if $ T_{i}>0 $ for some $ i<d $.
Finding solutions to (\ref{eq:rho}) in turn reduces to 
finding roots of
$P(\rho):=P_{1}(\rho)P_{2}(\rho)-P_{3}(\rho)P_{4}(\rho)$ 
with the polynomials $ P_{1}(\rho)  $, $ P_{2}(\rho)  $, $ P_{3}(\rho)  $, and $ P_{4}(\rho)  $ obtained by simplifying
\begin{align*}
R_{1}(\rho)&:=\sum_{i=1}^{d}\frac{\nu_{i}}{\phi^{2}_{m_{i}}(\rho)}T_{i} Q_{1}(\rho) (1-\rho)^{2},  & R_{2}(\rho)&:=\sum_{j=1}^{d_{0}-1}\frac{\alpha_{j}s_{j}}{\phi_{\alpha_{j}}(\rho)} Q_{2}(\rho), \\
R_{3}(\rho)&:=\sum_{i=1}^{d-1} \frac{\nu_{i}m_{{i}}}{\phi^{2}_{m_{i}}(\rho)} Q_{1}(\rho)  T_{i}, \quad \text{and} &
R_{4}(\rho)&:=\sum_{j=1}^{d_{0}}\frac{s_{j}}{\phi_{\alpha_{j}}(\rho)}Q_{2}(\rho)(1-\rho)^{2},
\end{align*}
respectively, where 
\[ 
Q_{1}(\rho) :=  \prod_{i=1}^{d-1} {\phi^{2}_{m_{i}}(\rho)}
\quad \text{and} \quad
Q_{2}(\rho) :=
\prod_{j=1}^{d_{0}-1}{\phi_{\alpha_{j}}(\rho)}.
\]
It can be seen that if $ \phi_{\tau}(\rho) \neq 0 $, where $ \tau \in \{m_{1},m_{2},\ldots,m_{d} \} \cup 
 \{s_{1},s_{2},\ldots,s_{d_{0}} \}  $, and $ P(\rho)=0 $, then $ \rho $
 is a solution to the equation (\ref{eq:rho}). Let us note that all solutions to (\ref{eq:rho}) can be obtained in this way.

It can be shown that
\begin{thm} \label{thm:degreeP}
\begin{enumerate}[(a)]
\item 
The degree of $ P(\rho) $ is bounded from above by $2d+d_{0}-4$.
\item 
If the condition (\ref{cond:ML_prob}) holds, then $ P(\rho) $ is a non-zero polynomial with probability $ 1 $.
\end{enumerate}
\end{thm}

\begin{proof}
The proof easily follows from Theorem 2.5  in \cite{Grz14}.
\end{proof}

The above theorem  suggests that the ML degree of the linear mixed model with two variance components does not exceed $ 2d+d_{0}-4 $ if the appropriate assumptions are satisfied. To prove this, we need the following

\begin{lem} \label{lem:1}
The probability that the ML equations (\ref{eq:MLEquations}) have a solution of the form $s=(\theta,-\theta)$,
where $\theta$ belongs to the set of complex numbers, is equal to $0$.
\end{lem}

\begin{proof}
Upon substitution 
$ \theta $ for $\sigma_{1}^{2}$
and $- \theta $ for $\sigma_{2}^{2}$ the system
(\ref{eq:MLEquations}) reduces to the following set of equations:
\begin{align} \label{eq:MLEquations11}
	\begin{split}
		\sum_{i=1}^{d-1}\frac{\nu_{i}m_{i}}{(m_{i}-1)^{2}}T_{i} &=\theta \sum_{j=1}^{d_{0}-1}\frac{s_{j}\alpha_{j}}{\alpha_{j}-1},\\
		\sum_{i=1}^{d}\frac{\nu_{i}}{(m_{i}-1)^{2}}T_{i}&=\theta \sum_{j=1}^{d_{0}}\frac{s_{j}}{\alpha_{j}-1},\\
		\theta& \neq 0.
	\end{split}
\end{align}
Let us observe that if
\begin{equation} \label{expr:1}
\sum_{j=1}^{d_{0}-1}\frac{s_{j}\alpha_{j}}{\alpha_{j}-1}=0,
\end{equation}
then the left-hand side of the first equation in the system (\ref{eq:MLEquations11}) is equal to $ 0 $. This implies that 
$T_{i}=0$, $i=1,2,\ldots,d-1$.
The probability of this event is equal to $ 0 $.

If the condition (\ref{expr:1}) does not hold, then 
\[
 \theta=\frac{\sum_{i=1}^{d-1}\frac{\nu_{i}m_{i}}{(m_{i}-1)^{2}}T_{i}}{\sum_{j=1}^{d_{0}-1}\frac{s_{j}\alpha_{j}}{\alpha_{j}-1}}
\]
can be regarded as a function of $ T_{1},\ldots,T_{d-1} $ and  
the second equation of the system (\ref{eq:MLEquations11}) can be rewritten in the form
\begin{equation} \label{eq:2-10new}
\sum_{i=1}^{d}\frac{\nu_{i}}{(m_{i}-1)^{2}} T_{i}  -  \theta 
\sum_{j=1}^{d_{0}}\frac{s_{j}}{\alpha_{j}-1 }=0.
\end{equation}

The left-hand side of the above equation can be expressed   
as the sum of the independent random variables $X_{1}$ and 
$X_{2}$ given by
\[
X_{1}=\frac{\nu_{d}}{(m_{d}-1)^{2}}T_{d}, \quad 
X_{2}=\sum_{i=1}^{d-1}\frac{\nu_{i}}{(m_{i}-1)^{2}} T_{i}  -  \theta 
\sum_{j=1}^{d_{0}}\frac{s_{j}}{\alpha_{j}-1 }.
\]
Since $X_{1}$ is absolutely continuous, 
the sum $X_{1}+X_{2}$  has a density (so it cannot have an atom in $0$) and this completes the proof.
\end{proof}

An immediate consequence of
Lemma \ref{lem:1}, Theorem \ref{thm:degreeP}, and the fact that $ T_{i}>0 $ outside a null set $(i=1,\ldots,d) $
is
\begin{thm} \label{thm:MLdegree}
If the model (\ref{model}) satisfies the condition (\ref{cond:ML_prob}), then its ML degree is 
bounded from above by
$ 2d+d_{0}-4$.
\end{thm}

\section{The REML degree of the model}
\label{sec:REML_Degree}

The REML equation system,
that is the ML equation system in the model 
$\mathcal{N}\{z,\mathbf{0}_{n-p},
\sigma_{1}^{2} BVB^{\prime}+ \sigma_{2}^{2}I_{n-p}\}$ with $z=BY$,
where $B$ is a matrix satisfying the condition (\ref{eq:B}),
can be presented in the form
\begin{align} \label{eq:REMLquations}
\begin{split}
  \sum_{i=1}^{d-1}\frac{\nu_{i}m_{i}}{(m_{i}\sigma_{1}^{2}+\sigma_{2}^{2})^{2}}T_{i} &=\sum_{j=1}^{d-1}\frac{\nu_{j}m_{j}}{m_{j}\sigma_{1}^{2}+\sigma_{2}^{2}},\\
\sum_{i=1}^{d}\frac{\nu_{i}}{(m_{i}\sigma_{1}^{2}+\sigma_{2}^{2})^{2}}T_{i}  &=\sum_{j=1}^{d}\frac{\nu_{j}}{m_{j}\sigma_{1}^{2}+\sigma_{2}^{2}},
\end{split}
\end{align}
see \cite[p.~291]{ GSTU02}. 
If $ T_{i}>0 $ for some $ i<d $, then
finding all the solutions to the system (\ref{eq:REMLquations}) that don't have the form $ s=(\theta,-\theta)^{\prime} $, where $ \theta $ is a complex number, can be reduced to finding all roots of the 
polynomial $P^{*}(\rho)$ obtained by simplifying
the algebraic expression $P_{0}^{*}(\rho):=(R_{1}^{*}(\rho)-R_{2}^{*}(\rho)) Q_{1}^{*}(\rho)$, where 
\begin{align}
R_{1}^{*}(\rho):=&\sum_{i=1}^{d}\frac{\nu_{i}}{\phi^{2}_{m_{i}}(\rho)}T_{i} \sum_{j=1}^{d-1}\frac{m_{j}\nu_{j}}{\phi_{m_{j}}(\rho)},   \\
R_{2}^{*}(\rho):=&\sum_{i=1}^{d-1} \frac{\nu_{i}m_{{i}}}{\phi^{2}_{m_{i}}(\rho)}T_{i}\sum_{j=1}^{d}\frac{\nu_{j}}{\phi_{m_{j}}
(\rho)} \quad
\text{and} \quad Q_{1}^{*}(\rho) :=  \prod_{i=1}^{d} {\phi^{2}_{m_{i}}(\rho)}         
\end{align}   
in a similar way as in the case of the ML equations.



\begin{thm} 

\begin{enumerate}[(a)]
\item The degree of the  polynomial $ P^{*}(\rho)$ is bounded from above by $2d-3 $.
\item If the condition (\ref{cond:REML_P1}) holds, then  $ P^{*}(\rho)$ is a non-zero polynomial with probability $ 1 $.
\end{enumerate}
\end{thm}

\begin{proof}
The theorem easily follows from Theorem 3.1 in \cite{Grz14}.
\end{proof}

The REML degree of the linear mixed model
$\mathcal{N}(Y, X \beta, \Sigma(s))$
can be defined as the ML degree of the model $\mathcal{N}(BY, \mathbf{0}_{n-p} , B\Sigma(s))B^{\prime})$, where $ B $ is a matrix satisfying (\ref{eq:B}).

\begin{thm} \label{thm:REMLDegree}
If the model (\ref{model}) satisfies the condition (\ref{cond:REML_P1}), then its REML degree does not exceed $ 2d-3 $.
\end{thm}
We omit the proof of Theorem \ref{thm:REMLDegree} since it is similar to that of
Theorem~\ref{thm:MLdegree}.

\section{General mean structure 
in the one-way layout}
\label{section:GeneralMean}

Let us consider the 
following extension of the
one-way random effects model: 
\begin{equation} \label{eq:OneWay}
	Y=W\beta+Z\alpha+\epsilon,
\end{equation}
where $\beta \in \mathbb{R}^p$ is a fixed mean parameter vector,
$\alpha=(\alpha_{1},\ldots,\alpha_{q}        )^{\prime}$
with $q \ge 2$
is the vector of random effects,
$ W $ is an $ n \times p $ matrix of rank $ p<n $ such that 
$ \mathbf{1}_{n} \in \mathcal{M}(W) $
and  
\begin{equation} \label{eq:Z}
Z= \left[\begin{array}{cccc}
              \mathbf{1}_{n_{1}} & \mathbf{0}_{n_{1}}   & \cdots   &  \mathbf{0}_{n_{1}} \\
             \mathbf{0}_{n_{2}}  &   \mathbf{1}_{n_{2}} &  \cdots  & \mathbf{0}_{n_{2}}  \\
                  \vdots   & \vdots &  \ddots        & \vdots \\
               \mathbf{0}_{n_{q}}  & \mathbf{0}_{n_{q}} & \cdots &  \mathbf{1}_{n_{q}}
         \end{array}\right], 
\end{equation}
where $n=\sum_{k=1}^{q}n_{k}$. We assume that $\alpha \sim N(\mathbf{0}_{q},\mathcal\sigma_{1}^{2}I_{q})$ and
$\epsilon \sim N(\mathbf{0}_n,\sigma_{2}^{2}I_{n})$ are independent. The model
was considered e.g.
in \cite[Section~5]{GDP12}.
It can be
expressed in the form (\ref{model})
with $ X=W $ and 
$ V=ZZ^{\prime} $.

To find  upper bounds for the ML degree and the REML degree of this model we will need the following

\begin{prop} \label{prop:matrices}
In the model (\ref{model}) with  $ X=W $ and $ V=ZZ^{\prime} $:

\begin{enumerate}[(a)]
\item The number of distinct eigenvalues of the matrix $ V $ does not exceed~$ q+1 $.
\item The number of distinct eigenvalues of the matrix $ BVB^{\prime} $ does not exceed~$ q $.
\end{enumerate}
\end{prop}

\begin{proof}
The proof of the part (a) follows from the equalities
\[ rank(V)=rank(ZZ^{\prime})=rank(Z)=q. \]
For the proof of (b) it is sufficient to note that
$ MVM=\sum\limits_{i=0}^{d}m_{i}B^{\prime}E_{i}B,
$
where $ m_{i} $ and $ E_{i} $, are as in  (\ref{eq:BVBprime}), see \cite[p.~285]{GSTU02},
and
\[ 
rank(BVB^{\prime})=rank(MVM)=rank(MZ)<rank(Z)=q.
\] 
\end{proof}

Now we are ready to state the following
\begin{thm} \label{Th:ConjectureML} 

\begin{enumerate}[(a)]
Consider the model (\ref{model})
with $ X=W $ and $ V=ZZ^{\prime} $.
\item If the condition (\ref{cond:ML_prob}) is satisfied, then
the ML degree of the model 
does not exceed $3q-3$.
\item 
If the condition (\ref{cond:REML_P1}) is satisfied, then
the REML degree of the model
 does not exceed $2q-3$.
\end{enumerate}

\end{thm}
\begin{proof}
The part (a) of the theorem is a consequence of Theorem \ref{thm:MLdegree} 
and Proposition \ref{prop:matrices} while the  part (b) follows from Theorem \ref{thm:REMLDegree} and  Proposition \ref{prop:matrices}.
\end{proof}

We have thus proved
that both parts of Conjecture~1 in the paper of Gross \textit{et al.} \cite{GDP12}, which 
concerns the maximum likelihood degree and the REML degree of 
the model (\ref{eq:OneWay}),
are true under fairly mild conditions.

\begin{remk}
The condition (\ref{cond:ML_prob}) as well
as the condition (\ref{cond:REML_P1}) are satisfied if $n_{k}>1$ for some $k$. The above theorem can be thus considered as a generalization of the results concerning the bounds for the ML degree and the REML degree for the one-way random effects model presented in \cite[p.~996]{GDP12}. 
\end{remk}

\section{Conclusion} \label{sec:Conclusion}
The recent research confirms that the likelihood function and the REML likelihood function may have multiple local maxima, see e.g.
\cite[pp.~2296--2297]{LBH15}
or \cite[Section~7]{GDP12}. The results obtained in this paper indicate that the approach proposed in the paper of Gross \textit{et al.} \cite{GDP12},
in which all critical points of the log-likelihood function
are found by solving a system of algebraic equations, may prove to be efficient for linear mixed models with two variance components.

\section*{Acknowledgements}
I would like to thank
Elizabeth Gross and
Mathias Drton for their help.


\begin{thebibliography}{11}
	
	\bibitem[\protect\citeauthoryear{Catanese et~al.}{2006}]{CHKS06}
	\begin{barticle}[author]
		\bauthor{\bsnm{Catanese},~\bfnm{F.}\binits{F.}},
		\bauthor{\bsnm{Ho\c{s}ten},~\bfnm{S.}\binits{S.}},
		\bauthor{\bsnm{Khetan},~\bfnm{A}\binits{A.}} \AND
		\bauthor{\bsnm{Sturmfels},~\bfnm{S.}\binits{S.}}
		(\byear{2006}).
		\btitle{The Maximum Likelihood Degree}.
		\bjournal{Amer. J. Math.}
		\bvolume{128}
		\bpages{671--697}.
		\bmrnumber{2230921}
	\end{barticle}
	\endbibitem
	
	\bibitem[\protect\citeauthoryear{Demidenko and Massam}{1999}]{DM99}
	\begin{barticle}[author]
		\bauthor{\bsnm{Demidenko},~\bfnm{E.}\binits{E.}} \AND
		\bauthor{\bsnm{Massam},~\bfnm{H.}\binits{H.}}
		(\byear{1999}).
		\btitle{On the existence of the maximum likelihood estimate in variance
			components models.}
		\bjournal{Sankhy\=a Ser. A}
		\bvolume{61}
		\bpages{431--443}.
		\bmrnumber{1743550}
	\end{barticle}
	\endbibitem
	
	\bibitem[\protect\citeauthoryear{Drton, Sturmfels and Sullivant}{2009}]{DSS08}
	\begin{bbook}[author]
		\bauthor{\bsnm{Drton},~\bfnm{M.}\binits{M.}},
		\bauthor{\bsnm{Sturmfels},~\bfnm{B.}\binits{B.}} \AND
		\bauthor{\bsnm{Sullivant},~\bfnm{S.}\binits{S.}}
		(\byear{2009}).
		\btitle{Lectures on Algebraic Statistics}.
		\bpublisher{Birkh{\"a}user Verlag}, \baddress{Basel}.
		\bmrnumber{2723140}
	\end{bbook}
	\endbibitem
	
	\bibitem[\protect\citeauthoryear{Gnot et~al.}{2002}]{GSTU02}
	\begin{barticle}[author]
		\bauthor{\bsnm{Gnot},~\bfnm{S.}\binits{S.}},
		\bauthor{\bsnm{Stemann},~\bfnm{D.}\binits{D.}},
		\bauthor{\bsnm{Trenkler},~\bfnm{G.}\binits{G.}} \AND
		\bauthor{\bsnm{Urbańska-Motyka},~\bfnm{A.}\binits{A.}}
		(\byear{2002}).
		\btitle{Maximum likelihood estimation in mixed normal models with two variance
			components}.
		\bjournal{Statistics}
		\bvolume{36}
		\bpages{283--302}.
		\bmrnumber{1923467}
	\end{barticle}
	\endbibitem
	
	\bibitem[\protect\citeauthoryear{Gross, Drton and Petrovi\'{c}}{2012}]{GDP12}
	\begin{barticle}[author]
		\bauthor{\bsnm{Gross},~\bfnm{E.}\binits{E.}},
		\bauthor{\bsnm{Drton},~\bfnm{M.}\binits{M.}} \AND
		\bauthor{\bsnm{Petrovi\'{c}},~\bfnm{S.}\binits{S.}}
		(\byear{2012}).
		\btitle{Maximum likelihood degree of variance component models}.
		\bjournal{Electron. J. Stat.}
		\bvolume{6}
		\bpages{993--2016}.
		\bmrnumber{2988436}
	\end{barticle}
	\endbibitem
	
	\bibitem[\protect\citeauthoryear{Grządziel}{2014}]{Grz14}
	\begin{barticle}[author]
		\bauthor{\bsnm{Grządziel},~\bfnm{M.}\binits{M.}}
		(\byear{2014}).
		\btitle{On maximum likelihood estimation in mixed normal models with two
			variance components}.
		\bjournal{Discuss. Math. Probab. Stat.}
		\bvolume{34}
		\bpages{187--197}.
		\bmrnumber{3286181}
	\end{barticle}
	\endbibitem
	
	\bibitem[\protect\citeauthoryear{Ho\c{s}ten, Khetan and
		Sturmfels}{2005}]{HKS05}
	\begin{barticle}[author]
		\bauthor{\bsnm{Ho\c{s}ten},~\bfnm{S.}\binits{S.}},
		\bauthor{\bsnm{Khetan},~\bfnm{A.}\binits{A.}} \AND
		\bauthor{\bsnm{Sturmfels},~\bfnm{B.}\binits{B.}}
		(\byear{2005}).
		\btitle{Solving the Likelihood Equations}.
		\bjournal{Found. Comput. Math.}
		\bvolume{5}
		\bpages{389--407}.
		\bmrnumber{2189544}
	\end{barticle}
	\endbibitem
	
	\bibitem[\protect\citeauthoryear{Jiang}{2007}]{Jia07}
	\begin{bbook}[author]
		\bauthor{\bsnm{Jiang},~\bfnm{Jiming}\binits{J.}}
		(\byear{2007}).
		\btitle{Linear and Generalized Linear Mixed Models and Their Applications}.
		\bpublisher{Springer}, \baddress{New York}.
		\bmrnumber{MR2308058}
	\end{bbook}
	\endbibitem
	
	\bibitem[\protect\citeauthoryear{Lavine, Brew and Hodges}{2015}]{LBH15}
	\begin{barticle}[author]
		\bauthor{\bsnm{Lavine},~\bfnm{M.}\binits{M.}},
		\bauthor{\bsnm{Brew},~\bfnm{A.}\binits{A.}} \AND
		\bauthor{\bsnm{Hodges},~\bfnm{J.}\binits{J.}}
		(\byear{2015}).
		\btitle{Approximately exact calculations for linear mixed models}.
		\bjournal{Electron. J. Stat.}
		\bvolume{9}
		\bpages{2293–-2323}.
		\bmrnumber{3411230}
	\end{barticle}
	\endbibitem
	
	\bibitem[\protect\citeauthoryear{Olsen, Seely and Birkes}{1976}]{OSB76}
	\begin{barticle}[author]
		\bauthor{\bsnm{Olsen},~\bfnm{A.}\binits{A.}},
		\bauthor{\bsnm{Seely},~\bfnm{J.}\binits{J.}} \AND
		\bauthor{\bsnm{Birkes},~\bfnm{D.}\binits{D.}}
		(\byear{1976}).
		\btitle{Invariant quadratic estimation for two variance components}.
		\bjournal{Ann. Statist.}
		\bvolume{4}
		\bpages{878--890}.
		\bmrnumber{418345}
	\end{barticle}
	\endbibitem
	
	\bibitem[\protect\citeauthoryear{Rao and Kleffe}{1988}]{RK88}
	\begin{bbook}[author]
		\bauthor{\bsnm{Rao},~\bfnm{C.~R.}\binits{C.~R.}} \AND
		\bauthor{\bsnm{Kleffe},~\bfnm{J.}\binits{J.}}
		(\byear{1988}).
		\btitle{Estimation of Variance Components and Applications}.
		\bpublisher{North Holland}, \baddress{Amsterdam}.
		\bmrnumber{0933559}
	\end{bbook}
	\endbibitem
	
\end{thebibliography}
\end{document}